\documentclass[11pt,a4paper,leqno]{article}
\usepackage{amsmath}
\usepackage{amsfonts}
\usepackage{amssymb}
\usepackage{amsthm}
\usepackage{graphicx}
\usepackage{enumitem}
\usepackage{hyperref}
\hypersetup{
	colorlinks=true,
	linkcolor=blue,
	anchorcolor=black,
	filecolor=blue,
	citecolor = red,   
	menucolor=.,   
	urlcolor=cyan,
	breaklinks=true,
}

\usepackage[margin=2.5cm]{geometry}

\newtheorem{theorem}{Theorem}[section]
\newtheorem{corollary}[theorem]{Corollary}
\newtheorem{lemma}[theorem]{Lemma}
\newtheorem{defn}{Definition}

\newtheorem{claim}{Claim}[section]

\theoremstyle{remark}
\newtheorem*{remark}{Remark}
\newcommand{\D}{\mathbb{D}}
\newcommand{\C}{\mathbb{C}}
\title{One component bounded functions}

\author{Carlo Bellavita and Artur Nicolau 
\thanks{The second author is supported in part by the Generalitat de Catalunya (grant 2017 SGR 395), the Spanish Ministerio de Ciencia e Innovaci\'on (project  MTM2017-85666-P) and the Spanish Research Agency (Mar\'ia de
Maeztu Program CEX2020-001084-M )} }
\date{}
\newcommand{\Addresses}{{% additional braces for segregating \footnotesize
  \bigskip
  \footnotesize

  C.~Bellavita, \textsc{Department of Mathematics, Università degli Studi di Milano, Milan, Italy}\par\nopagebreak
  \textit{E-mail address}, C.~Bellavita: \texttt{carlo.bellavita@unimi.it}

  \medskip

  A.~Nicolau %(Corresponding author)
  , \textsc{Departament de  
Matem\`atiques, Universitat Aut\`onoma de Barcelona and Centre de Recerca Matem\`atica, 08193 Barcelona}\par\nopagebreak
  \textit{E-mail address}, A.~Nicolau: \texttt{artur@mat.uab.cat}

}}

\begin{document}
\maketitle

\begin{abstract}
\noindent Three different characterizations of one-component bounded analytic functions are provided. The first one is related to the the inner-outer factorization, the second one is in terms of the size of the reproducing kernels in the corresponding de Branges-Rovnyak spaces and the last one concerns the associated Clark measure.
% \noindent \textbf{Keywords}. Inner function, bounded mean oscillation, Hardy spaces, Dirichlet class. 
\end{abstract}

%%%%%%%%%%%%%%%%%%%%%%%%%%%%%%%%%%%%%%%%%%%%%%%%%%%%%%%%%%%%%%%%%%%%
%%%%%%%%%%%%%%%%%%%%%%%%%%%%%%%%%%%%%%%%%%%%%%%%%%%%%%%%%%%%%%%%%%%%
\section{Introduction}
\label{sec1}

Let $\D$ be the open unit disc of the complex plane. Let ${H}^{\infty}(\mathbb{D})$ be the algebra of bounded analytic functions in~$\mathbb{D}$ and let $\|b\|_{\infty}=\sup \{|b(z)|: z\in \mathbb{D}\}$ for $b \in {H}^{\infty}(\mathbb{D}) $. An inner function is a bounded analytic function in $\D$ whose radial limits have modulus one at almost every point of the unit circle $\partial \D$.   
% It is a classical result that any inner function can be factorized as the product of a Blaschke product, a singular inner function and unimodular constant (\cite{Garnett1981}).
% Recall that, for a given sequence $\{z_n\}\subset\D$ satisfying $\sum_n (1-|z_n|)<\infty$, the Blaschke product
% with zeros $\{z_n\}$ is defined by
%    \begin{equation*}\label{Eq:Blaschke}
%     B(z)=\prod_{n}\frac{|z_n|}{z_n}\frac{z_n-z}{1-\overline{z}_nz}, \quad z\in \D.
%     \end{equation*}
% Here each zero is repeated according to its multiplicity and the convention $|z_n|/z_n=1$ is used when $z_n=0$.
% A singular inner function is an inner function of the form
%     $$
%     S(z)=\exp\left(\int_{\partial \D} \frac{z+\xi}{z-\xi}\, d\s(\xi) \right),\quad z\in\D,
%     $$
% where $\s$ is a positive measure on $\partial \D$, singular with respect to the Lebesgue measure. The singular set of an inner function $\Th$, which will be denoted by $\sing \Th$, consists of all points on $\partial \D$ in which $\Th$ does not have an
% analytic continuation. If $\Th$ factors as $\Th = \lambda BS$, where $|\lambda|=1$, $B$ is a Blaschke product and $S$ is the singular inner function associated to the singular measure $\sigma$, then $\sing \Th$ is precisely the union of the accumulation points of zeros of $\Th$ and the closed support of the measure $\sigma$. See Chapter II of \cite{Garnett1981}.  
An inner function $b$ is called one-component if there exists a number $0<C<1$ such that the sublevel set
    $\{z\in \D:|b(z)|<C \}$
is connected. One-component inner functions were introduced by B. Cohn in \cite{C} in connection with Carleson measures for model spaces and have been already widely studied. Indeed, as subsequently proved by A. Volberg and S.Treil in \cite{VT} and by A. Aleksandrov in \cite{A2}, it is possible to provide an useful geometric characterization of Carleson measures for model spaces corresponding to one-component inner functions. Other descriptions of one-component inner functions have been given in \cite{A3}, \cite{NR} and \cite{Be}, algebraic properties of this family have been studied in \cite{CM1}, \cite{R} and \cite{CM2}, and further properties of the corresponding model spaces have been given in \cite{A1}, \cite{Ba}, \cite{Be} and \cite{BBK}.

\vspace{11 pt}\noindent
This paper is devoted to the study of one-component bounded functions. We start with the definition of this family. Given $b \in H^\infty(\mathbb{D})$ with $\|b\|_{\infty}=1$, its spectrum is defined as  
$$
\text{spec}(b):=\left\lbrace \xi \in \partial \mathbb{D}: \ \liminf_{z \to \xi }|b(z)|<1\right\rbrace\ .
$$
%of $H^\infty(\mathbb{D})$ functions.
\begin{defn}\label{def1.1}
Let $b \in H^\infty(\mathbb{D})$ with $\|b\|_{\infty}=1$. Then, $b$ is called one-component if there exists a constant $0<C<1$ such that the following two conditions hold  
\begin{enumerate}[label=(\alph*)]
\item $\text{spec}(b) \subseteq \overline{ \left\lbrace z \in \mathbb{D}: |b(z)|\leq C\right\rbrace}$ , 
\item The set $\{z \in \mathbb{D}:\  |b(z)|<C\}$ is connected.
\end{enumerate}
\end{defn}

\noindent
One-component bounded functions were introduced by A. Baranov, E. Fricain and J. Mashreghi in \cite{BFM}, where the authors studied Carleson measure for de Branges-Rovnyak spaces. They provided an useful characterization for such measures for de Branges-Rovnyak spaces corresponding to one-component bounded functions. Note that one-component inner functions are one-component bounded functions because if $b$ is inner, spec$(b) = \left\{ \xi \in \partial \mathbb{D}: \ \liminf_{z \to \xi }|b(z)|=0\right\}$ (see \cite{Ga}). 
Therefore, when $b$ is inner, Condition \emph{a)} of Definition \ref{def1.1} is always satisfied. 

\vspace{11 pt}\noindent
In this paper, we provide several characterizations of  one-component bounded functions, which are analogous to the already known results holding for one-component inner functions. 
We will extend to this bigger family the description of \cite{NR} in terms of the positions of the zeros, of \cite{A2} in terms of the growth of the reproducing kernels of the associated de Branges-Rovnyak spaces and, finally, of \cite{Be} in terms of the expression of the associated Clark measures. It is worth mentioning that our proofs use some ideas from these previous papers but some new arguments are also required to deal with more general situation.

\vspace{22 pt}\noindent
Given $b \in H^\infty(\mathbb{D})$, as already done in \cite{Bi}, we consider the measure $\sigma_b$ with support in $\overline{\mathbb{D}}$ associated to $|b(z)|$. Indeed, if $b=B_{\{z_n\}}S_\mu O_b$ is the classical factorization of $b$ in a Blaschke product $B_{\{z_n\}}$ with zeros $\{ z_n \}$, a singular inner function $S_\mu$ and an outer factor $O_b$, we define
$$
d \sigma_b(z):=\sum_n\left(1-|z_n|\right)\delta_{z_n}(z)+d\mu(z)+\log|b(z)|^{-1}dm(z) , \quad z \in \overline{\D} , 
$$
where $dm$ is the Lebesgue measure on $\partial \mathbb{D}$.
 We recall also that if $z \in \mathbb{D} \setminus \{0 \}$, the associated Carleson square $Q(z)$ is defined as
$
Q(z):=\left\lbrace w \in \mathbb{D}: \ |z|<|w|<1,\ |w/|w|-z/|z||<(1-|z|)/2\right\rbrace\ .
$ Given a Carleson square $Q(z)$, $z \in \mathbb{D}  \setminus \{0 \}$, we denote its sidelenght by $\ell(Q(z)) = 1- |z|$ and its top part by $T(Q(z)):= \left\lbrace w \in  Q(z) : |z|<|w|<(1+|z|)/2\right\rbrace.
$ Our first description of one-component bounded functions uses the measure $\sigma_b$ and it is analogous to the main result in \cite{NR}.

\begin{theorem}\label{theo1.1}
Let $b \in H^\infty(\mathbb{D})$ with $||b||_{\infty}= 1$. Then, $b$ is one-component if and only if there exists $0<C_1<1$ such that
\begin{equation}\label{eq1.1}
    \sigma_b\left( Q(z)\right)=0 \quad\text{whenever}\quad |b(z)|\geq C_1\ .
\end{equation}
\end{theorem}\noindent 
It will be shown that the spectrum of a one-component bounded function must be a proper closed subset of the unit circle. As an application of Theorem \ref{theo1.1}, we obtain the following converse result.  

\begin{theorem}\label{theo1.2}
Let $E  \subsetneq \mathbb{T}$ be a closed set. Then there exists a one-component bounded function $b$ such that spec$(b)=E$.
\end{theorem}\noindent
We point out that, when $m(E)=0$, it is always possible to find a one-component inner function whose spectrum is equal to $E$ (\cite{NR}). 

\vspace{11 pt}\noindent
Our second characterization of one-component bounded functions   uses the reproducing kernels of the associated de Branges-Rovnyak spaces. In this paper we will not work with the 
de Branges-Rovnyak spaces and we will not introduce them. The interested reader may deepen this subject looking at \cite{FM2}. The only fact one needs to keep in mind, is that the de Branges-Rovnyak spaces are reproducing kernel Hilbert spaces of holomorphic functions in the unit disk $\mathbb{D}$. Let $b \in H^\infty(\mathbb{D})$ with $||b||_{\infty}= 1$. The reproducing kernel of $H(b)$ at $a \in \mathbb{D}$, is 
$$
k^b_a(z):=\frac{1-\overline{b(a)}b(z)}{1-\overline{a}z}\ , \qquad z \in \mathbb{D} , 
$$
that is, $\left\langle f,k^b_a\right\rangle_{H(b)}=f(a)$, for every $f \in H(b)$. If the function $b$ is inner, the associated de Branges-Rovnyak space is the classical model space. We next describe one-component bounded functions in terms of the size of $k^b_a(z)$.
\begin{theorem}\label{theo1.3}\
\begin{enumerate}
\renewcommand{\labelenumi}{\alph{enumi})}
\item Let $b$ be a one-component bounded function. Then there exists a constant $C=C(b) >0$ such that for every $a \in \mathbb{D}$ we have
\begin{equation}\label{eq1.2}
    \sup_{z\in \mathbb{D}}\bigg| \frac{1-\overline{b(a)}b(z)}{1-\overline{a}z}\bigg| \leq C  \frac{1-|b(a)|^2}{1-|a|^2}\ .
\end{equation}
\item Let $b \in H^\infty(\mathbb{D})$ with $\|b \|_{\infty} = 1$. Assume there exists a constant $0<C<1$ such that  
\begin{equation}\label{eq1.3}
    \text{spec}(b) \subseteq \overline{\left\lbrace z\in \mathbb{D}: |b(z)|\leq C\right\rbrace}\ .
\end{equation}
If 
\begin{equation}\label{eq1.4}
    \sup_{z\in \mathbb{D}}\bigg| \frac{1-\overline{b(a)}b(z)}{1-\overline{a}z}\bigg|=o\left(  \frac{1}{1-|a|^2}\right)\ , \text{ as } |b(a)| \to 1,
\end{equation}
then $b$ is one-component.
\end{enumerate}
\end{theorem}\noindent
In Theorem \ref{theo1.3} we have considered the reproducing kernels $k^b_a(z)$ at a point $a \in \mathbb{D}$. In the following results, we deal also with $a \in \partial\mathbb{D}$. It is easy to see that any one-component bounded function $b$ extends analytically across $ \partial \mathbb{D}\setminus$ spec$(b)$, see \cite{FM}. 

\begin{theorem}\label{theo1.4}\
\begin{enumerate}
\renewcommand{\labelenumi}{\alph{enumi})}
\item Let $b$ be a one-component bounded function. Then there exists a constant $C=C(b) >0$ such that
\begin{equation}\label{eq1.5}
    \frac{|b(z)-b(\xi)|}{|z-\xi|} \leq C \big| b'(\xi)\big|\ , \qquad   z \in \mathbb{D} , \, \xi \in \partial \mathbb{D}\setminus spec(b)\ .  
\end{equation}
\item Let $b \in H^\infty(\mathbb{D})$ with $\|b \|_{\infty} = 1$. Assume there exists $0<C<1$ such that spec$(b)\subset \overline{\left\lbrace z \in \mathbb{D}: \ |b(z)|\leq C\right\rbrace}$. If estimate (\ref{eq1.5}) holds,  the function $b$ is one-component.
\end{enumerate}
\end{theorem}
\noindent
Theorem \ref{theo1.4} is applied to obtain the following result.

\begin{theorem}\label{theo1.5}\
\begin{enumerate}
\renewcommand{\labelenumi}{\alph{enumi})}
\item Let $b$ be a one-component bounded function. Then the  function
\begin{equation}\label{eq1.6}
    H(\xi):=\begin{cases}1/b'(\xi)&\ \text{ if } \xi \notin \text{spec}(b)\\
    0 &\ \text{ if } \xi \in \text{spec}(b)
    \end{cases}
\end{equation}
is Lipschitz.
\item Let $b \in H^\infty(\mathbb{D})$ with $\|b \|_{\infty} = 1$. Assume there exists $0<C<1$ such that spec$(b)\subset \overline{\left\lbrace z \in \mathbb{D}: \ |b(z)|\leq C\right\rbrace}$. If the function $H$ defined in (\ref{eq1.6}) is Lipschitz, then $b$ is one-component.
\end{enumerate}
\end{theorem}\noindent
Theorem \ref{theo1.5} is also equivalent to the following corollary. 

\begin{corollary}\label{cor1.6}
Let $b \in H^\infty(\mathbb{D})$ with $\|b \|_{\infty} = 1$. Then  $b$ is one-component if and only if the following three conditions hold: 
\begin{enumerate}
\renewcommand{\labelenumi}{\alph{enumi})}
\item There exists a constant $0<C<1$ such that 
$
\text{spec}(b)\subset \overline{ \left\lbrace z \in \mathbb{D}: \ |b(z)|\leq C\right\rbrace}\ .
$
\item There exists a constant $C_1 >0$ such that  $|b''(\xi)|\leq C_1 |b'(\xi)|^2$, for every  $\xi \in \partial \mathbb{D}\setminus $ spec$(b)$
\item $|b'(\xi)|\to \infty$, as $0<$dist$(\xi, $spec$(b))\to 0$.
\end{enumerate}
\end{corollary}\noindent
These last three results are similar to their analogue versions for one-component inner functions, presented in \cite{A1} and \cite{A3}. However, we point out that our proofs are different. Indeed, unlike \cite{A3}, we do not make use of the explicit expression of the norms of the reproducing kernels in de Branges-Rovnyak spaces.

\vspace{11 pt}\noindent
We describe also one-component bounded functions in terms of their Clark measures. Clark measures of one-component inner functions were described in \cite{Be} by R.V. Bessonov. In the following, we give the corresponding result for one-component bounded functions.
Before stating the result, we recall the definition and main properties of Clark measures. For a more complete description we refer to \cite{PS}, \cite{S} and \cite{CMR}. 

Let $b \in H^\infty(\mathbb{D})$ with $\|b \|_{\infty} = 1$ and $|\alpha|=1$. Since $(\alpha + b) / (\alpha - b)$ is an analytic function in $\D$ having positive real part, there exists a unique positive measure $\mu_\alpha$ in the unit circle such that
\begin{equation}\label{eq1.7}
    \frac{\alpha+b(z)}{\alpha-b(z)}=\int_{\partial \mathbb{D}} \frac{\xi+z}{\xi-z}d\mu_\alpha(\xi) +2i\frac{\Im\left( \bar{\alpha}b(0)\right)}{|\alpha-b(0)|^2}\ , \quad z \in \D. 
\end{equation}
The measure $\mu_\alpha$ is called the Clark measure of the function $b$ at the value $\alpha$. 

We note that if $b(0)=0$ the measure $\mu_\alpha$ is a probability measure. It follows from (\ref{eq1.7}) that the real part of $(\alpha+b(z))/(\alpha-b(z))$ is the Poisson integral of the measure $\mu_\alpha$. The measure  $\mu_\alpha$ is singular if and only if the function $b$ is inner. In this case,  the measure $\mu_\alpha$ is carried by $\left\lbrace \xi \in \partial \mathbb{D}: \ b(\xi)=\alpha\right\rbrace$. 

For every positive, Borel measure $\mu$ on the unit circle $\partial\mathbb{D}$ we denote by $a(\mu)$ the set of all its isolated atoms. The set $P(\mu):=$ spt$(\mu )\setminus a(\mu)$ consists of all the accumulation points in the support of $\mu$, denoted by spt$(\mu)$. We will say that an atom $\xi \in a(\mu)$ has two neighbours if there is an open arc $(\xi_-,\xi_+)$ on the unit circle $\partial\mathbb{D}$ with endpoints $\xi_\pm \in a(\mu)$ such that $(\xi_{-} , \xi_+) \cap $ spt$(\mu)=\{\xi\}$.

\begin{theorem}\label{theo1.6}
Let $\mu$ be a positive, Borel measure on $\partial \mathbb{D}$.
The following two condition are equivalent:
\begin{enumerate}
\renewcommand{\labelenumi}{\alph{enumi})}
\item$\mu$ is the Clark measure of a one-component bounded function.
\item $\mu=\mu_a+\mu_s$ and there exists a constant $C=C(\mu) >0$ such that 
\begin{enumerate} [label=\roman{*}., ref=(\roman{*})]
    \item The measure $\mu_a$ is absolutely continuous and   
    $\mu_a(\xi)=g(\xi)dm(\xi)$, with $1/C\leq g(\xi)\leq C$ for $\mu_a$-almost all $\xi \in \partial \D$.
    \item The measure $\mu_s$ is purely atomic and every atom has two neighbours. Moreover, there are infinitely many atoms in any connected component $I$ of $\partial \mathbb{D}\setminus P(\mu)$, which accumulate to both  boundary points of $I$. Finally, for any $\xi \in a(\mu)$ we have
    $$
    C^{-1}|\xi-\xi_\pm|\leq \mu \{\xi\} \leq C|\xi-\xi_\pm| \ .
    $$
    \item We have 
    $$
    H^*(\mu)(\xi):=\limsup_{\epsilon\to 0} \bigg| \int_{\{t \in \partial \D: |t-\xi|>\epsilon \} }\frac{d\mu(t) }{1-\bar{t}\xi}\bigg|\leq C
    $$
    for any $\xi \in$ spt$(\mu_a)\cup a(\mu)$.
 \end{enumerate}
\end{enumerate}

\end{theorem}\noindent
\noindent
 It is worth mentioning that the conditions $C^{-1}\leq g(\xi)\leq C$ for $\mu_a$-almost every point $\xi \in \partial \D$ and $C^{-1}|\xi-\xi_\pm|\leq \mu(\xi)\leq C|\xi-\xi_\pm|, \,   \xi \in a(\mu), $ can be rephrased in the following way: There exists a constant $C_1>0$ such that $C_1^{-1}|I|\leq \mu(I)\leq C_1|I|$ for every arc $I \subset \partial \mathbb{D}$ containing at least two different points of spt$(\mu)$. It is also interesting to mention that in the statement of Theorem 1.7, condition \emph{b.iii)} can be replaced by the apparently stronger condition that $H^* (\mu) (\xi) \leq C$ at any point $\xi \in \text{spt}(\mu)$.

This paper is divided in five short sections. Next section contains some auxiliary results. Section 3 is devoted to the proofs of Theorems \ref{theo1.1}, \ref{theo1.2} and some consequences. Theorems \ref{theo1.3}, \ref{theo1.4} and \ref{theo1.5} are proved in Section 4. Last Section contains the proof of Theorem \ref{theo1.6}.

%%%%%%%%%%%%%%%%%%%%%%%%%%%%%%%%%%%%%%%%%%%%%%%%%%%%%%%%%%%%%%%%%%%%
%%%%%%%%%%%%%%%%%%%%%%%%%%%%%%%%%%%%%%%%%%%%%%%%%%%%%%%%%%%%%%%%%%%%
\newpage \noindent
\section{Preliminary results}\label{sec2}
In this section, we collect some auxiliary results that will be used in the rest of the paper. 

\vspace{11 pt}\noindent
We first give some comments about Definition \ref{def1.1}. As already mentioned, the main difference between one-component bounded functions and one-component inner functions is Condition \emph{a)} in Definition 1. Our first auxiliary result says that Condition \emph{a)} allows us to choose freely the level $C$ of the connected sublevel set.
\begin{lemma}\label{lem2.1}
Let $b \in H^\infty (\mathbb{D})$ with $\|b \|_\infty = 1$. Assume there exists a constant $0< C < 1$ such that $\Omega_C:=\{ z \in \mathbb{D}: |b(z)|<C\}$ is connected and spec$(b)\subset \overline{\Omega_C}$. Then  $\Omega_{C_1}:=\{ z \in \mathbb{D}: |b(z)|<{C_1}\}$ is connected for every $C_1>C$.
\end{lemma}
\begin{proof}
We argue by contradiction and we assume that there exists $C<C_1 < 1$ such that $\Omega_{C_1}$ is not connected. Note that $\Omega_{C} \subset \Omega_{C_1}$. Let $\Omega_1$ be a connected component of $\Omega_{C_1}$ such that $\Omega_1 \cap \Omega_C=\emptyset$. Note that $\Omega_1$ is simply connected. Let $\phi : \mathbb{D}\to \Omega_1$ be a conformal map. Then the function 
$ U= b\circ \phi / C_1$ is inner. Since $\inf_{z \in \mathbb{D}}|U(z)|>\delta>0$, the function $U$ has to be constant, which is impossible. This contradiction finishes the proof. 
\end{proof}\noindent
There are some differences between the spectrum of an inner function and the spectrum of a bounded analytic function. Let us consider a bounded function $b$, which, according to its inner-outer factorization, can be written as $b= B S_\mu O_b $. Therefore, as stated in \cite{BFM},  $\overline{\text{spec}(b)}$ is the smallest closed subset of $\partial \mathbb{D}$ containing the limit points of the zeros of the Blaschke product $B$ and the supports of the measures $\mu$ and $\log|b|^{-1}dm$. It is well known and easy to prove that $b$ has an analytic extension and it is unimodular through any arc of the open set $\partial\mathbb{D}\setminus \overline{\text{spec}(b)}$, see \cite{FM}. 
Note that there are $b \in  H^\infty (\mathbb{D})$ whose spectrum is not closed. 

\begin{lemma}\label{lem2.2}
Let $b \in  H^\infty (\mathbb{D})$ with $\|b\|_{\infty} = 1$. Assume there exists a constant $0<C<1$ such that $\text{spec}(b) \subset \overline{ \{z \in D : |b(z)| \leq C\} } $. Then, spec$(b)$ is closed.
\end{lemma}
\begin{proof}
Let us consider a sequence $\{\xi_n\}\subset \text{spec}(b)$ such that $\xi_n\to \xi$. For every $n$ the exists a sequence $\{t_{n,k}\}\subset \mathbb{D}$ such that $t_{n,k}\to \xi_n$ as $k \to \infty $ and $ |b(t_{n,k})|\leq  C $. If we apply a diagonal argument to the sequences $\{t_{n,k}\}$, we find a sequence $\{t_m\}\subset \mathbb{D}$ such that $t_m\to \xi$ and $|b(t_m)|\leq C$, which implies also that $\xi \in $spec$(b)$.
\end{proof}\noindent
So, the spectrum of a one-component bounded funcion is a closed subset of the unit circle.  Given a closed connected set $E \subset \overline{\D}$ let $w(a, \Gamma, \D \setminus E)$ be the harmonic measure at the point $a \in \D \setminus E$ of the set $\Gamma \subset \partial (\D \setminus E)$ in the domain $\D \setminus E$, that is, the value at the point $a$ of the harmonic function in the domain $\D \setminus E$ whose boundary values are $1$ at $\Gamma$ and $0$ elsewhere. We will use the following consequence of the classical Hall's Lemma. We recall that $\rho(z,w)$ is the pseudo-hyperbolic distance between $z$ and $w$.

\begin{lemma}\label{lem2.3}
Let $\Gamma$ be a curve contained in the unit disc joining the origin with a point in a Carleson square $Q(z)$ with $|z| >1/2$. Assume that $\rho (z, \Gamma) \geq 1/2$ and  $\D \setminus \Gamma$ is connected. Then there exits a constant $C(\Gamma) >0$ such that $w(z, \Gamma, \D \setminus \Gamma) \geq C(\Gamma) $.
\end{lemma}\noindent
\begin{proof}
Let $\tau$ be the automorphism of the unit disc with $\tau^{-1} = \tau$ and $\tau (z) =0$. The conformal invariance of harmonic measure gives $w(z, \Gamma, \D \setminus \Gamma) = w(0, \tau(\Gamma), \D \setminus \tau (\Gamma))$. Note that $\tau ( \Gamma)$ is a curve satisfying $\inf\{|w|: w \in \tau (\Gamma)\} \geq 1/2$ 
and having diameter bounded below by an absolute constant. Thus, either the radial projection of $\tau (\Gamma)$ onto the unit circle or its circular projection onto the unit interval, is a connected set whose diameter is bounded below by an absolute constant. Then the radial or circular version of Hall's Lemma (see page 124 of \cite{GM}) finish the proof. \end{proof} 
\noindent 
Lemma \ref{lem2.3} will be used to prove the following auxiliary result .
\begin{lemma}\label{lem2.4}
Let $b$ be a one-component bounded function. Then there exists a constant $0<C<1$ such that $\limsup_{r\to 1} |b(r\xi)|\leq C<1$, for every $\xi \in $ spec$(b)$. 
\end{lemma}
\begin{proof}
Let $\xi \in $ spec$(b)$. Since $b$ is one-component, there exists a constant $0<C_1 <1$ and a curve $\Gamma \subset \D$ connecting the origin with $\xi$ such that $|b(w)| < C_1$ for any $w \in \Gamma$. The Maximum Principle gives that $$
\log |b(z)| \leq (\log C_1) w (z, \Gamma , \D \setminus \Gamma ), \quad z \in \D \setminus \Gamma . 
$$
Hence Lemma \ref{lem2.3} gives that $ |b(r \xi)| \leq C_1^{C(\Gamma)}$ if $\rho(r \xi, \Gamma) \geq 1/2$. Note that if $\rho(r \xi , \Gamma) < 1/2$, Schwarz's Lemma gives $|b(r \xi)| \leq (C_1 + 1/2)/ (1+ C_1 / 2)$. This finishes the proof. 
% it follows that  $\limsup_{r\to 1} |b(r\xi)|\leq C_1^{1/2}$ (see page 305 of \cite{Tsuji}). 
\end{proof}
\noindent
Since one-component bounded functions $b$ satisfy $\|b\|_{\infty} = 1$, Lemma \ref{lem2.4} gives that the spectrum of a one-component bounded functions is a proper subset of the unit circle.

Before concluding this preliminary section and moving to the proofs of the main results, we recall some properties of the derivatives of bounded analytic functions at points in the complement of their spectrum. For a more complete discussion on this subject, we refer to \cite{FM}.

\begin{lemma}\label{lem2.5}
Let $b \in H^\infty(\mathbb{D})$ with $||b||_{\infty} = 1$. Then, $ |b'(r\xi)|\leq 4|b'(\xi)|$ for every $\xi \in \partial \mathbb{D}\setminus \overline{\text{spec}(b)}$ and $1>r>1/2$.
\end{lemma}
\begin{proof}
Without loss of generality we can assume $\xi = 1$. Let $b =B_{\{z_n\}}S_\mu O_b$ be the inner-outer factorization of $b$. It is sufficient to prove the lemma separately for the three factors. Firstly
$$
|B'(r)|\leq \sum_n \frac{1-|z_n|^2}{|1-\overline{z_n}r|^2}\leq \frac{1}{r^2}\sum_n \frac{1-|z_n|^2}{|1/r-\overline{z_n}|^2}\leq 4 \sum_n \frac{1-|z_n|^2}{|1-\overline{z_n}|^2}=4|B'(1)|\ .
$$
For the singular inner factor we use the estimate $|\xi - r|^2 \geq |\xi - 1|^2 / 2$, $\xi \in \partial \D$, $1/2 < r < 1$. Then 
$$
|S_{\mu}'(r)|\leq 
\int_{\partial\mathbb{D}}\frac{2}{|\xi-r|^2}d\mu(\xi)\leq 
\int_{\partial\mathbb{D}}\frac{4}{|\xi-1|^2}d\mu(\xi)=2|S'_\mu(1)| .
$$
The same argument also gives 
$$
|O_b'(r)|\leq \bigg| \int_{\partial\mathbb{D}}\frac{2\xi}{(\xi-r)^2}\log|b^{-1}(\xi)|dm(\xi) \bigg|\leq 
\int_{\partial\mathbb{D}}\frac{4}{|\xi-1|^2}\log|b^{-1}(\xi)|dm(\xi)=2|O_b'(1)| \ .
$$
This finishes the proof.
\end{proof}

\begin{lemma}\label{lem2.6}
Let $b \in H^\infty(\mathbb{D})$ with $||b||_{\infty}=1$. Let $\xi_0 \in \partial \mathbb{D}$ such that $b(z)$ extends analytically at a neighbourhood $\mathcal{U}$ of $\xi_0$, having values of modulus one on $\partial \mathbb{D}\cap \mathcal{U}$. Then 
$$
\frac{\xi b'(\xi)}{b(\xi)}>0
$$
for every $\xi \in \mathcal{U}\cap \partial \mathbb{D}$.
\end{lemma}
\begin{proof}
 Let $b =B_{\{z_n\}}S_\mu O_b$ be the inner-outer factorization of $b$. With a straight computation we obtain that 
$$
\frac{\xi b'(\xi)}{b(\xi)}=\sum_n \frac{1-|z_n|^2}{|\xi-z_n|^2}+\int_{\partial\mathbb{D}}\frac{2}{|\xi-t|^2}d\mu(t)+\int_{\partial\mathbb{D}}\frac{2}{|\xi-t|^2}\log|b'(t)|^{-1}dm(t)\ , \quad \xi \in \partial \mathbb{D}\cap \mathcal{U}, 
$$
from which the lemma follows.
\end{proof}

%%%%%%%%%%%%%%%%%%%%%%%%%%%%%%%%%%%%%%%%%%%%%%%%%%%%%%%%%%%%%%%%%%%%
%%%%%%%%%%%%%%%%%%%%%%%%%%%%%%%%%%%%%%%%%%%%%%%%%%%%%%%%%%%%%%%%%%%%
\vspace{22 pt}\noindent
\section{Inner-Outer Factorization}\label{sec3}

In this section we will prove Theorem \ref{theo1.1} and Theorem \ref{theo1.2}, providing also some easy corollaries.  A Carleson square $Q$ is called dyadic if its closure intersected with the unit circle is a dyadic arc of the unit circle. Note that each dyadic Carleson square contains two dyadic Carleson squares of half sidelength. 

\begin{proof}[Proof of Theorem \ref{theo1.1}]
We define $\Omega_K:=\lbrace z \in \mathbb{D}:\ |b(z)|\leq K\rbrace $, $K>0$. 
We first prove that condition (\ref{eq1.1}) is necessary. Pick a constant $0 < C_1 <1$ such that $|b(0)|<C_1<1$, $\Omega_{C_1}$ is connected and
\begin{equation}\label{eq3.2}
\left\lbrace \xi \in \partial \mathbb{D}\ : \ \liminf_{z \to \xi}|b(z)|<1\right\rbrace \subseteq \overline{\Omega_{C_1}}\ .    
\end{equation} We argue by contradiction. Assume that there exists a sequence of points $\lbrace z_n\rbrace \subset \mathbb{D}$ such that $\lim_{n \to \infty}|b(z_n)|=1$ and $\sigma_b(Q(z_n))>0$ for every $n$. Note that $\Omega_{C_1} \cap Q(z_n)=\emptyset$ if $n$ is sufficiently large, since, otherwise, there would exists a connected path $\Gamma$ joining the origin and a point in $Q(z_n)$ with $|b(w)|\leq {C_1}$ for any point $w \in \Gamma$. Schwarz's Lemma gives that $\rho (z_n , \Gamma) \to 1$ as $n \to \infty$. The Maximum Principle gives that 
$$
\log |b (z)| \leq (\log C_1) w (z, \Gamma, \D \setminus \Gamma), \quad z \in \D \setminus \Gamma .  
$$
Now Lemma \ref{lem2.3} gives that $|b(z_n)|<C_1^{C( \Gamma)}<1$ if $n$ is sufficiently large,  which contradicts the assumption. Therefore, if $n$ is sufficiently large the set $\Omega_{C_1} \cap Q(z_n)$ has to be empty and, in particular,  $b(z)\neq 0$ for every $z \in Q(z_n)$. On the other hand, because of (\ref{eq3.2}), we deduce also that $\lim_{z\to \xi}|b(z)|=1$, for every $\xi \in \overline{Q(z_n)}\cap \partial \mathbb{D}$, which implies that $\sigma_b(Q(z_n))=0$. 

We now prove that condition (\ref{eq1.1}) is sufficient. Let $C_1<C_2<1$ be a constant which will be fixed later. Let $\mathcal{A}=\lbrace Q_j\rbrace$ be the collection of maximal dyadic Carleson squares such that
\begin{equation}\label{eq3.3}
\sup_{z \in T(Q_j)} |b(z)|\geq C_2\ .
\end{equation}
By maximality, $|b(w)|\leq C_2$ for any $w \in \mathbb{D}\setminus \cup_j Q_j$. We note that if $(1-C_2)(1-C_1)^{-1}$ is sufficiently small, Schwarz's Lemma shows that Condition (\ref{eq3.3}) implies that $|b(z)|\geq C_1$ for any $z \in T(2Q_j)$. We fix $0 < C_2 <1$ with this property. Therefore, the assumption (\ref{eq1.1}) implies that $\sigma_b(2Q_j)=0$. In particular, $b$ extends analytically through $\overline{2Q_j}\cap \partial \mathbb{D}$ and $|b(\xi)|=1$ for every $\xi \in \overline{2Q_j}\cap \partial \mathbb{D}$. Hence, we deduce that
$$
\left\lbrace \xi \in \partial \mathbb{D}\ : \ \liminf_{z \to \xi}|b(z)|<1\right\rbrace \subseteq \overline{\mathbb{D}\setminus \cup Q_j} \cap \partial\mathbb{D}\ .
$$
We fix $C_3>C_2$. We note that $\Omega_{C_3}\supseteq \mathbb{D}\setminus \cup Q_j$. Hence 
$$
\left\lbrace \xi \in \partial \mathbb{D}\ : \ \liminf_{z \to \xi}|b(z)|<1\right\rbrace \subseteq \overline{\Omega_{C_3}}\ .
$$
It only remains to show that $\Omega_{C_3}$ is connected. Let $\Omega_1$ be the connected component of $\Omega_{C_3}$ containing $\mathbb{D}\setminus \cup Q_j$. If,  by contradiction, we assume that there exists another connected component $\Omega_2 \neq \Omega_1$ of $\Omega_{C_3}$, then $\Omega_2 \subset \cup Q_j$. Note that $\Omega_2$ is simply connected and $\partial \Omega_2\cap \partial\mathbb{D}$ can have at most two points. If $ \phi: \ \mathbb{D}\to \Omega_2$ is a conformal map, then $(C_3)^{-1} ( b \circ \phi) $ would be an inner function. Since it cannot be constant, we deduce that $\inf_{\Omega_2} |b(z)|=0$. However, this would contradict the following claim.
\begin{claim}\label{cl3.1}
There exists a constant $C>0$ such that for any $j=1,2, \ldots$ and any $w \in Q_j = Q(z_j)$ we have 
$$
C^{-1} \frac{\log |b(z_j)|^{-1}}{1-|z_j|} \leq \frac{\log |b(w)|^{-1}}{1-|w|} \leq C \frac{\log |b(z_j)|^{-1}}{1-|z_j|} . 
$$
\end{claim}
\begin{proof}
We note that 
$$
\log|b(w)|^{-1}\cong\int_{\overline{\mathbb{D}}}P_w(z)d\sigma_b(z)\ ,
$$
when $w \in Q(z_j)$. Since $\sigma_b(2Q(z_j))=0$, we deduce
$$
\frac{\log |b(w)|^{-1}}{1-|w|} \cong \int_{\overline{\mathbb{D}}\setminus 2Q(z_j)}\frac{1}{|1-\bar{z}w|^2}d\sigma_b(z)\cong \int_{\overline{\mathbb{D}}\setminus 2Q(z_j)}\frac{1}{|1-\bar{z}z_j|^2}d\sigma_b(z)\cong\frac{\log |b(z_j)|^{-1}}{1-|z_j|} \ .
$$
\end{proof}
\noindent 
This claim finishes the proof.
\end{proof}
\noindent
Using the Theorem \ref{theo1.1}, it is easy to verify that the product of two one-component bounded functions is still one-component. 
\begin{corollary}\label{cor3.1}
Let $b_1 , b_2$ be two one-component bounded functions. Then  $b_1 b_2$ is a one-component bounded function.
\end{corollary}
\begin{proof}
We need to verify the sufficient condition (\ref{eq1.1}) of Theorem \ref{theo1.1}. We know that there exists a constant $0<C<1$ such that 
$$
\sigma_{b_i}(Q(z))=0, \ \text{ if } |b_i(z)|\geq C , \quad i=1,2.
$$
Pick a constant  $C < C_1 < 1$. Assume $|b_1(z)b_2(z)|>C_1$. Then  $|b_i(z)|>C_1$, $i=1,2$ and hence 
$$
\sigma_{b_1 b_2}(Q(z))=\sigma_{b_1}(Q(z))+\sigma_{b_2}(Q(z))=0\ . 
$$
This finishes the proof.
\end{proof}
\noindent
We note that the same result for one-component inner functions has been already proved in \cite{CM1} with a different argument. Theorem \ref{theo1.1} can be also used to prove that any proper closed subset $E \subset \partial \D$ is the spectrum of  one-component bounded function.

\begin{proof}[Proof of Theorem \ref{theo1.2}]
Write $\partial \mathbb{D}\setminus E= \cup I_k$, where $I_k$ are open arcs. For any $I_k$ consider its endpoints $e^{i t_k(1)} $, $e^{i t_k (2)}$ and the {\it isosceles triangle} $T_k$ defined as
$$ 
T_k = \left\{r e^{i t} : t_k (1) < t < t_k (2) , 0 \leq 1-r< \min \{ |t - t_k (1)|, |t - t_k (2)| \} \right\}  . 
$$
We locate $\{z_{k,j}\}_j$ on $\partial T_k \cap \mathbb{D}$ such that the pseudo-hyperbolic distance $\rho(z_{k,j+1},z_{k,j})=\delta$ is independent of $k$ and $j$ and $\delta < 1/2$. The sequence $\{z_{k,j}\}_{j,k}$ is a Blaschke sequence since 
$$
\sum_k\sum_j 1-|z_{k,j}|\lesssim\sum_k |I_k|< \infty\ .
$$
Let $B$ be the Blaschke product with zeros $\{z_{k,j} : k, j\}$.  Consider the bounded analytic function $b$ defined as
$$
b(z):=\exp \left( \int_E \frac{\xi+z}{\xi-z}\log(1/2) dm(\xi)\right) B(z) , \quad z \in \D .
$$
Since $\partial \D \setminus E$ has positive measure we have $\|b\|_\infty = 1$. Because of the construction, it is clear that spec$(b)=E$. The only thing left to prove is that the function $b$ is one-component. We first note that if $z \in \cup_k \partial T_k \cap  \D$, Schwarz's lemma gives 
$ |B(z)|\leq 1/2$. Hence $|b(z)| < 1/2$ for any $z \in  \cup_k \partial T_k \cap  \D \cup E$. The Maximum Principle gives that $|b(z)|\leq 1/2$ for any $z \in \mathbb{D}\setminus \cup_k T_k$. Fix $C>1/2$.  If $|b(z))|\geq C$ then $z\in \cup_k T_k$ and consequently $\sigma_b(Q(z))=0$. We can now apply Theorem \ref{theo1.1} and we deduce that $b$ is one-component.

%\begin{figure}
%\includegraphics[scale=0.3]{TEO3.2.png}
%\end{figure}
\end{proof}
\noindent
We close this section with another consequence of Theorem \ref{theo1.1}. 
\begin{corollary}\label{cor3.3}
Let $b \in H^\infty ( \D)$ with $\|b\|_\infty = 1$. Assume there exists a constant $0<C<1$ such that spec$(b)\subset \overline{\left\lbrace z \in \mathbb{D}:\ |b(z)|\leq C\right\rbrace}$. Then there exists a Blaschke product $B$ such that the function $bB$ is one-component and spec$(b)=$spec$(bB)$.
\end{corollary}
\begin{proof}
Lemma \ref{lem2.2} gives that spec$(b)$ is closed. We write $\partial \mathbb{D}\setminus \text{spec}(b)=\cup_k I_k$ where each $I_k$ is a closed arc of $\partial \mathbb{D}$ satisfying $|I_k| = \text{dist}(I_k, \overline{\Omega_C})$. 
%We fix a sequence $\{\epsilon_k\} \subset (0, 1)$ such that  $\lim_{k\to \infty} \epsilon_k = 0$.
We choose $r_k \in (0,1)$ such that, if $|z|\geq r_k$ and $e^{i\text{arg}(z)} \in I_k$, then $|b(z)|\geq (1+ C)/2$. Let $\Gamma$ be the region given by $ \Gamma = \cup_{k=1}^\infty \{ r\xi: r\geq r_k \text{ and } \xi \in I_k\}$. We fix $0< \varepsilon < C$ and we pick points $\{z_{j} \}\subset \partial \Gamma \cap \mathbb{D}$ so that the pseudo-hyperbolic distance $ \rho (z_{j} , z_{j+1})=\varepsilon$. We consider the Blaschke product $B$ whose zeros are $\{z_j\}$. Due to Schwarz's lemma, $|B(z)|\leq C$
when $z \in \partial\Gamma\cap \mathbb{D}$ and, consequently, $|bB| < C$ on $ \left(\partial \Gamma \cap \D\right)    \cup \text{spec}(b)$. The Maximum Principle gives $|b(z)B(z)| \leq C$ for any $z \in \mathbb{D} \setminus \Gamma$. We fix $C< C_1 < 1 $. If $|b(z)B(z)|\geq C_1$ then $z\in \Gamma$ and consequently $\sigma_{bB}(Q(z))=0$. We apply now Theorem \ref{theo1.1} and we obtain that $bB$ is one-component.
\end{proof}
\noindent
Before concluding this section, we highlight that if $b_1, b_2 $ are one-component bounded functions, even if $b_1 /b_2 \in H^\infty$, it may happen that $b_1 /b_2 $ is not one-component. For sake of completeness, we recall that the regularity of the quotient of two one-component inner functions has been studied in \cite{CM2}.

%%%%%%%%%%%%%%%%%%%%%%%%%%%%%%%%%%%%%%%%%%%%%%%%%%%%%%%%%%%%%%%%%%%%
%%%%%%%%%%%%%%%%%%%%%%%%%%%%%%%%%%%%%%%%%%%%%%%%%%%%%%%%%%%%%%%%%%%%
\vspace{22 pt}
\section{Reproducing kernels for one-component bounded function}

In this section we will prove Theorems \ref{theo1.3}, \ref{theo1.4} and \ref{theo1.5}.

\begin{proof}[Proof of Theorem \ref{theo1.3}]
We start by proving \emph{a)}. Let $0< C_1 < 1$ be the constant given by Theorem \ref{theo1.1} and let $C_1< C_2 = C_2 (C_1) < 1$ be a constant to be fixed later. We consider the family of maximal dyadic Carleson squares $\left\lbrace Q(z_j)\right\rbrace$ such that 
$$
\sup_{z \in T(Q(z_j))} |b(z)|\geq C_2 . 
$$
We note that $|b(z)|\leq C_2$ when $z \in \mathbb{D}\setminus \cup_j Q(z_j)$. Taking $C_2 > C_1$ sufficiently close to $1$, Schwarz's lemma gives that  $\sup_{z \in T(\widetilde{Q}(z_j))}|b(z)|\geq C_1$, where $\widetilde{Q}(z_j)$ is the dyadic Carleson square containing $Q(z_j)$ with double sidelenght. We fix $0< C_2 < 1$ with this property. Applying Theorem \ref{theo1.1}, we obtain that $\sigma_b(\widetilde{Q}(z_j))=0$. Therefore, the function $b$ extends analytically on $\overline{\widetilde{Q}(z_j)}\cap \partial \mathbb{D}$ and $|b(\xi)|=1$ for every $\xi\in \overline{\widetilde{Q}(z_j)}\cap \partial \mathbb{D}$. Claim \ref{cl3.1} gives  
\begin{equation}\label{eq4.1}
\frac{\log|b(w)|^{-1}}{1-|w|}\cong \frac{\log|b(z_j)|^{-1}}{1-|z_j|}\cong \frac{1}{|\ell(Q(z_j))|}\ , 
\end{equation}
for every $w \in Q(z_j)$, uniformly on $j$.

With no loss of generality, we assume that $|b(a)|$ is close to $1$. Hence $a \in \cup_j Q(z_j)$. We fix $j$ so that $a \in Q(z_j)$. Moreover, we can also assume that the pseudo-hyperbolic distance $\rho(a, \mathbb{D}\setminus \cup_j Q(z_j))$ is close to 1, since, otherwise, by Schwarz's lemma $|b(a)|$ would not be close to $1$. We want to show that there exists a constant $C>0$ such that
$$
\bigg| \frac{1-\overline{b(z)}b(a)}{1-\overline{z}a}\bigg|\leq C \frac{1-|b(a)|^2}{1-|a|^2}\ , \quad z \in \mathbb{D} . 
$$
If $z \notin 2Q(z_j)$, then
$$
\bigg| \frac{1-\overline{b(z)}b(a)}{1-a\bar{z}}\bigg| \leq \frac{2}{|\ell(Q(z_j))|}\cong \frac{\log|b(a)|^{-1}}{1-|a|}\cong \frac{1-|b(a)|^2}{1-|a|^2}\ .
$$
If $z\in 2Q_j$, since
$$
\frac{1-\overline{b(z)}b(a)}{1-a\bar{z}}=\frac{1-|b(a)|^2}{1-a\bar{z}}-b(a)\frac{\overline{b(z)-b(a)}}{1-\overline{z}a}\ 
$$
and 
$$
\frac{1-|b(a)|^2}{|1-\overline{z}a|}\lesssim \frac{1-|b(a)|^2}{1-|a|^2}\ ,
$$
it is sufficient to prove that
$$
\bigg| \frac{b(z)-b(a)}{1-\overline{z}a}\bigg|\lesssim \frac{1-|b(a)|^2}{1-|a|^2}\ .
$$
For $w \in \D \setminus \{0 \}$, we denote $w^* = w / |w|$.  We note that
$$
\frac{| b(z)-b(a)|}{|1-\overline{z}a|}\leq \frac{| b(z)-b(z^*)|}{|1-\overline{z}a|}+\frac{|b(z^*) -b( a^*)|}{|1-\overline{z}a|}+\frac{| b(a^*)-b(a) |}{|1-\overline{z}a|}\ ,
$$
and we estimate the three terms separately. Firstly, we note that (\ref{eq4.1}) gives that 
\begin{equation}\label{eq4.1.2}
    |b'(\xi)|\lesssim \frac{1}{\ell(Q(z_j))}, \quad \xi \in  \partial\mathbb{D}  \cap \overline{Q(z_j)}
\end{equation}
and consequently
$$
\frac{|b(z^*)-b (a^*)|}{|1-\overline{z}a|}\leq \frac{1}{|1-\overline{z}a|}\int_{z^*}^{a^*}|b'(t)|dt\lesssim \frac{C}{|\ell(Q(z_j))|}\frac{|z-a|}{|1-\overline{z}a|}\lesssim \frac{1-|b(a)|^2}{1-|a|^2}\ ,
$$
where in the last inequality we have used (\ref{eq4.1}).
Moreover, by using Lemma \ref{lem2.5}, (\ref{eq4.1.2}) and (\ref{eq4.1}), we obtain that
$$
\frac{|b(a)-b(a^*)|}{|1-\overline{z}a|}\lesssim \frac{1-|a|}{|\ell(Q(z_j))|}\frac{1}{|1-\overline{z}a|}\cong \frac{1-|b(a)|^2}{1-|a|^2}
$$
which proves the statement \emph{a)}.

We prove statement \emph{b)} by contradiction. We assume there exist points $a_n \in \mathbb{D}$ with $|b(a_n)|\to 1$, but $\sigma_b(Q(a_n))>0$. Let $0<C<1$ be a constant to be fixed later and let $\Omega_C=\lbrace z \in \mathbb{D}:\ |b(z)|<C\rbrace $. The assumption (\ref{eq1.4}) gives that $Q(a_n)\cap \Omega_C=\emptyset$ for $n$ large enough. Since spec$(b)\subseteq \overline{\Omega_C}$, we deduce also that spec$(b)\cap \overline{Q(a_n)}=\emptyset$. Hence $\sigma_b(Q(a_n))=0$ which is a contradiction.
\end{proof}

\begin{remark}\label{rem4.1}
We note that condition (\ref{eq1.3}) of Theorem \ref{theo1.3} is necessary. Indeed, $b(z)=(1+z)/2$ does not satisfy (\ref{eq1.3}), and consequently it is not one-component, while (\ref{eq1.4}) holds.
\end{remark}
\noindent
We move now to the proof of Theorem \ref{theo1.4}.
\begin{proof}[Proof of Theorem \ref{theo1.4}]
We start by proving \emph{a)}. We apply Theorem \ref{theo1.3} and obtain a constant $C>0$ such that 
$$
\bigg| \frac{1-\overline{b(a)}b(z)}{1-\overline{a}z}\bigg|\leq C \frac{1-|b(a)|^2}{1-|a|^2}\ , 
$$
for every $a,z \in \mathbb{D}$. For any $\xi \in \partial \mathbb{D}\setminus$spec$(b)$, we pick a sequence $\{a_n\}\subset \mathbb{D}$, which tends to $\xi$, and we obtain 
$$
\frac{|b(z)-b(\xi)|}{|z-\xi|}\leq C \big| b'(\xi)\big|, \, z \in \mathbb{D}\ .
$$

We move now to the proof of statement \emph{b)}. For sake of clarity, we split the argument in four claims. Since by Lemma \ref{lem2.2} the set spec$(b)$ is closed, we write $\partial \mathbb{D}\setminus$ spec$(b)=\cup_j I_j$, where $I_j$ are open arcs. Let $0<C_1 < 1 $ be a constant to be fixed later. For every $\xi \in I_j$, we choose $r(\xi)$ such that $1-r(\xi)=C_1/ |b'(\xi)|$. 

\begin{claim}\label{cl4.1}
There exists a constant $C_2 >0$ such that 
$$
\frac{1}{|b'(\xi)|}\leq C_2 \text{dist}(\xi,\text{spec}(b))\ , \quad \xi \in \partial\mathbb{D}\setminus \text{spec}(b) . 
$$
\end{claim}
\begin{proof}
For every $\eta \in$ spec$(b)$, we take $\{ z_n\} \subset \Omega_C$ approaching $\eta$. Since $1-C\leq |b(z_n)-b(\xi)|\lesssim |b'(\xi)||z_n -\xi|$, we deduce that
$$
\frac{1}{|b'(\xi)|}\lesssim |\eta -\xi|\ .
$$
\end{proof}
\begin{claim}\label{cl4.2}
There exist constants $C_3 , C_4 >0$ such that 
for every $\xi \in \partial \mathbb{D}\setminus $spec$(b)$ and for every  $|z-\xi|\leq C_3/|b'(\xi)|$, we have 
\begin{equation}\label{eq4.2}
|b(z)-b(\xi)-b'(\xi)(z-\xi)|\leq C_4 |z-\xi|^2|b'(\xi)|^2\ .
\end{equation}
\end{claim}
\begin{proof}
Fix $\xi \in \partial \mathbb{D}\setminus $spec$(b)$. Consider the disc $D(\xi) = \{z \in \C : |z-\xi|\leq C_3/{|b'(\xi)|} \}$. 
Due to the previous claim, the two functions in (\ref{eq4.2}) are analytic in $D(\xi)$. Note that by assumption, estimate (\ref{eq4.2}) holds when $z \in \partial D(\xi)$ and by the Maximum Principle the estimate (\ref{eq4.2}) holds for any $z  \in D(\xi)$.
\end{proof}
\begin{claim}\label{cl4.3}
There exists a constant $0< C_5 < 1$ such that $|b(r(\xi)\xi)|\leq C_5$ 
for every $\xi \in \partial \mathbb{D}\setminus $spec$(b)$.
\end{claim}
\begin{proof}
Since $|b(\xi)|=1$, we rewrite (\ref{eq4.2}) as
\begin{equation}\label{eq4.3}
\bigg| \frac{b(z)}{b(\xi)}-1-\frac{b'(\xi)}{b(\xi)}\xi\left(\frac{z}{\xi}-1\right)\bigg| \leq C_4 |z-\xi|^2|b'(\xi)|^2\ .
\end{equation}
We consider $z=r(\xi)\xi$. Then
$$
\bigg| \frac{b(r(\xi)\xi)}{b(\xi)}-1-\frac{b'(\xi)}{b(\xi)}\xi\left(r(\xi)-1\right)\bigg| \leq C_4 (1-r(\xi))^2|b'(\xi)|^2\ ,
$$
which implies that ${b(r(\xi)\xi)}/{b(\xi)}-1$ belongs to the disk $B$ with centre at the negative point $b'(\xi) \xi\left(r(\xi)-1\right))/b(\xi)$ with radius comparable to  $(1-r(\xi))^2|b'(\xi)|^2$. Since $C_1 < 1$, $B$ is entirely contained in the left half plane. Therefore, the point ${b(r(\xi)\xi)}/{b(\xi)}$ is contained in $1+B\subset \mathbb{D}$, which is separated from $\partial \mathbb{D}$. We deduce that there exists a constant $0< C_5 < 1$ such that $|b(r(\xi)\xi)|\leq C_5$ for every $\xi \in \partial \mathbb{D}\setminus$ spec$(b)$.
\end{proof}\noindent
The Claim \ref{cl4.3} proves that 
$$
|b|\leq C_5 \text{ on } \cup_j \left\lbrace r(\xi)\xi: \ \xi \in I_j\right\rbrace\ .
$$
Consider the domain $\Omega:=\left\lbrace r\xi:\ 0<r<1 \text{ if }\xi \in \text{spec}(b)\ , \ 0<r<r(\xi) \text{ if } \xi \in \cup_j I_j\right\rbrace$. 
By the Maximum Principle  
$$
|b|\leq \max \{C, C_5\}<1 \text{ on } \Omega . 
$$
\begin{claim}\label{cl4.4}
There exists a constant $C_6  >0 $ such that $|b(z)| > C_6$ 
for every $z \in \mathbb{D}\setminus \Omega$.
\end{claim}
\begin{proof}
By assumption, we know that
$$
|b(z)-b(\xi)|\leq C |b'(\xi)||z-\xi|\ . \quad z \in \D, \xi \in \partial \D \setminus  \text{spec}(b) . 
$$
Taking $0< C_1 < 1$ sufficiently small, we deduce that 
$$
|b(z)|\geq 1-C C_1 >0 \ , \quad  z \in \D \setminus \Omega . 
$$
\end{proof}
\noindent We are now ready to prove that the function $b$ is one-component. We fix a constant  $1>C_7>C_i$, $i=1, \ldots, 6$ and consider the sublevel set $\mathcal{W}:= \left\lbrace z\in \mathbb{D}:\ |b(z)|<C_7\right\rbrace$. Therefore $\Omega \subset \mathcal{W}$. We assume there exists a non empty connected component $\Omega^* $ of $ \mathcal{W} $ such that $\Omega^*\cap \Omega=\emptyset$. Then  $\Omega^*$ would be simply connected and contained in $\mathbb{D}\setminus \Omega$. Consider a conformal map $\phi:  \mathbb{D}\to \Omega^*$. Since  $\overline{\Omega^*}\cap \partial\mathbb{D}$ contains at most two points, the function $ b\circ \phi / C_7$ is inner. Claim \ref{cl4.4} gives that $|b(z)|>C_6$ and hence the inner function $ b\circ \phi / C_7$ is bounded from below. Therefore it has to be constant which is impossible. Consequently, $\mathcal{W}$ is connected, which concludes the proof.
\end{proof}
\noindent
We note that if $b$ is one-component, Theorem \ref{theo1.4} tells us that
\begin{equation}\label{eq4.7}
\sup_{z \in \mathbb{D}}  \frac{|1-\overline{b(\xi)}b(z)|}{|1-\bar{\xi}z|}  \leq C|b'(\xi)|, \quad \xi \in \partial \D \setminus \text{spec}(b) . 
\end{equation}
We now deduce the following useful estimate. 
\begin{corollary}\label{cor4.1}
Let $b\in H^\infty$ with $||b||_{\infty}=1$. Fix $0<C<1$ and consider the sublevel set $\Omega_C:=\left\lbrace z \in \mathbb{D}: |b(z)|\leq C\right\rbrace$. 
\begin{enumerate}
\renewcommand{\labelenumi}{\alph{enumi})}
\item There exist a constant $C_1 >0$  such that dist$(\xi, \Omega_C\  \cup$ spec$(b))\leq C_1/|b'(\xi)|,$ for any $\ \xi \in \partial\mathbb{D}\setminus$ spec$(b)$.
\item If $b$ is one-component and spec$(b)\subset \overline{\Omega_C}$, then there exits a constant $C_2 >0$ such that dist$(\xi, \Omega_C)\geq C_2/|b'(\xi)|,$ for any $\  \xi \in \partial\mathbb{D}\setminus$ spec$(b)$.
\end{enumerate}
\end{corollary}
\begin{proof}
The proof of statement \emph{a)} is contained in \cite{BFM}. On the other hand, statement \emph{b)} is an immediate consequence of (\ref{eq4.7}).
\end{proof}
\noindent
We note that when $b$ is one-component, $|b'(\xi)|$ is (uniformly) comparable to dist$(\xi, \Omega_C )$, $\xi \in \partial \mathbb{D}\setminus spec(b)$. Here $\Omega_C ={\left\lbrace z \in \mathbb{D}: \ |b(z)|\leq C\right\rbrace}  $ is the connected sublevel set of $b$. 

We are now ready to prove Theorem \ref{theo1.5}.
\begin{proof}[Proof of Theorem \ref{theo1.5}]
We start by proving statement \emph{a)}. Claim \ref{cl4.2} in the proof of Theorem \ref{theo1.4} provide constants $C_1 , C_2 >0$ such that 
$$
|b(z)-b(\xi)-b'(\xi)(z-\xi)|\leq C_1 |z-\xi|^2|b'(\xi)|^2
$$
when $\xi \in \partial \D \setminus$ spec$(b)$ and $|z-\xi|\leq C_2/|b'(\xi)|$. Dividing by $|z-\xi|^2$ and taking $z$ tending to $\xi$, we obtain 
$$
|b''(\xi)|\leq 2C_1|b'(\xi)|^2\ 
$$
for every $\xi \in \partial \D \setminus$ spec$(b)$. We note that Corollary \ref{cor4.1} tells us that for any $\eta \in \partial\left( \text{spec}(b)\right)$ we have 
$$
\lim_{\xi \to \eta, \xi \in \partial \mathbb{D}\setminus spec(b)} |b'(\xi)|=\infty\ .
$$
Therefore the function $H(\xi)$ defined in (\ref{eq1.6}) is Lipschitz.

For the proof of statement \emph{b)}, we use some ideas from the proof of Theorem \ref{theo1.4}. Due to Lemma \ref{lem2.2}, spec$(b)$ is a closed subset of $\partial \mathbb{D}$. We denote $\partial \mathbb{D}\setminus $ spec$(b) = \cup_j I_j$ and for every $\xi \in I_j$ consider $0 < r(\xi) <1$ defined by 
$$
1-r(\xi)=\frac{1}{16}\frac{1}{|b'(\xi)|}\ .
$$
Since the function $H$ is Lipschitz, there exists a constant $C>0$ such that 
$$
\bigg| \frac{1}{b'(\xi)}\bigg|\leq C|\xi-\eta|\ , \quad  \xi \in \partial \D \setminus \text{spec}(b),\ \eta \in \text{spec}(b)\ .
$$
Therefore 
$$
\bigg| \frac{1}{b'(\xi)}\bigg| \leq C \text{ dist}(\xi, \text{spec}(b))\ .
$$
Moreover, for every $t,\xi \in I_j$ we have
$$
\bigg| \frac{1}{b'(t)}-\frac{1}{b'(\xi)}\bigg| \leq C|t-\xi|
$$
which implies that
\begin{equation}\label{eq4.4}
\frac{1}{2}\bigg| \frac{1}{b'(\xi)}\bigg|\leq \bigg| \frac{1}{b'(t)}\bigg| \leq \frac{3}{2} \bigg| \frac{1}{b'(\xi)}\bigg| 
\end{equation}
when $|t-\xi|< (2C|b'(\xi)|)^{-1}$.
For $z, \xi \in I_j$,
$$
b(z)-b(\xi)-b'(\xi)(z-\xi)=\int_{\xi}^{z}b''(t)(z-t)dt \ .
$$
Since $|b''(t)|\leq C|b'(t)|^2$ for $t \notin $ spec$(b)$, we deduce that
$$
\big| b(z)-b(\xi)-b'(\xi)(z-\xi)\big| \leq C\int_{\xi}^{z}|b'(t)|^2|z-t|dt\ .
$$
Using (\ref{eq4.4}), we obtain a constant $C_1 >0$ such that 
\begin{equation}\label{eq4.5}
\big| b(z)-b(\xi)-b'(\xi)(z-\xi)\big|\leq C_1 |b'(\xi)|^2|z-\xi|^2\ , \quad z, \xi  \in I_j
\end{equation}
For sake of clarity, we split the rest of the proof in two claims.
\begin{claim}\label{cl4.4.1}
There exists $C_2<1$ such that $|b(r(\xi)\xi)|\leq C_2, \ \xi \in I_j$\ .
\end{claim}
\begin{proof}
We argue by contradiction. Assume there exist $\xi_n \in I_j$ such that $|b(r(\xi_n)\xi_n)|\to 1$. Since $||b||_\infty=1$ and
$$
b(r(\xi_n)\xi_n)=\int_{\partial \D} P_{r(\xi_n)\xi_n}(t)b(t)dm(t)\ ,
$$
for any $\delta>0$, we have

$$
\frac{1}{1-r(\xi_n)} m \left\lbrace t:\ |t-\xi_n|<1-r(\xi_n) \text{ with } |b(t)-b(r(\xi_n)\xi_n)|>\delta \right\rbrace \to 0\ , \quad \text{ as } n \to \infty . 
$$
This contradicts (\ref{eq4.5}) and finishes the proof of the Claim.
\end{proof}
\noindent We consider
$$
\Omega:=\left\lbrace r\xi: \xi \in \partial\mathbb{D}, 0<r<1 \text{ if } \xi \in \text{spec}(b), 1 \geq  1-r>C_2 / |b'(\xi)| \text{ if } \xi \in \partial \D \setminus  \text{spec}(b) \right\rbrace .
$$
We note that $|b(\xi)|\leq C$ for almost every $\xi \in$ spec$(b)$. On the other hand, Claim \ref{cl4.4.1} tells us that $|b(r(\xi)\xi)|\leq C_2$ for every $\xi \in \cup_j I_j$. Then, the Maximum Principle gives 
\begin{equation}\label{eq4.6}
\sup_\Omega |b(z)|\leq \max\left( C,C_2\right) <1\ .    
\end{equation}
\begin{claim}\label{cl4.5}
We have $|b(z)|\geq 1/2$ for any $z \in \mathbb{D}\setminus \Omega$.
\end{claim}
\begin{proof}
We fix $z \in \mathbb{D}\setminus \Omega$ and we consider $\xi=z/|z|$. There exists $0<r<1$ such that 
$$
|b(z)-b(\xi)| \leq |b'(r\xi)|(1-|z|)\ .
$$
We apply Lemma \ref{lem2.5} to deduce that
$$
|b(z)-b(\xi)|\leq 4 |b'(\xi)|(1-|z|)\ .
$$
Since $z \in \mathbb{D}\setminus \Omega$, we have $1-|z|<1-r(\xi)=1 /16\ |b'(\xi)|$ and  we obtain that
$$
|b(z)-b(\xi)|<1/2\ .
$$
We deduce that $|b(z)|\geq 1/2$, which proves the claim.
\end{proof}
\noindent
We can now conclude the proof of statement \emph{b)}.
We fix $1> C_3> \max \{C, C_1, C_2, 1/2 \}$ and we prove that $\Omega_{C_3} = \{z \in \D : |b(z)| < C_3\}$ is connected. We note that (\ref{eq4.6}) implies that $\Omega$ is contained in $\Omega_{C_3}$. Let $\Omega_1 \subset \mathbb{D}\setminus \Omega$ be another connected component of $\Omega_{C_3}$. Then $\Omega_1$ is simply connected and there exists $j$ such that 
$$
\Omega_1\subset \left\lbrace r\xi: \ \xi \in I_j, 1>r>r(\xi)\right\rbrace\ .
$$
Let $\phi: \mathbb{ D} \to \Omega_1$ be a conformal mapping. Since $\partial \Omega_1 \cap\partial \mathbb{D}$ contains at most two points, the function $U:=(C_3)^{-1}b\circ \phi $ is inner. However, Claim \ref{cl4.5} says that $U$ is bounded from below and this gives that $\Omega_1$ is empty, that is, $\Omega_{C_3}$ is connected. 
\end{proof}
\noindent
We finally prove Corollary \ref{cor1.6}.
\begin{proof}[Proof of Corollary \ref{cor1.6}]
We have just to prove that (\ref{eq1.6}) is equivalent to the conditions \emph{b)} and \emph{c)} of Corollary \ref{cor1.6}. Indeed, if \emph{b)} and \emph{c)} hold, the function $H$ defined in (\ref{eq1.6}) has to be Lipschitz because it has bounded derivative at almost every point of the unit circle. 

On the other hand, if $H$ is Lipschitz, then, due to Rademacher's theorem and the regularity of $b'(\xi)$ when $\xi\notin \text{spec}(b)$, we obtain condition \emph{b)}. Analogously, condition \emph{c)} follows directly from the definition of $H$. 
\end{proof}

%%%%%%%%%%%%%%%%%%%%%%%%%%%%%%%%%%%%%%%%%%%%%%%%%%%%%%%%%%%%%%%%%%%%%%%%%%%%%%
%%%%%%%%%%%%%%%%%%%%%%%%%%%%%%%%%%%%%%%%%%%%%%%%%%%%%%%%%%%%%%%%%%%%%%%%%%%%%%
\vspace{22 pt}\noindent
\section{Clark measures of one-component bounded function}\label{sec5}

In this section, we prove Theorem \ref{theo1.6}, characterizing Clark measures of one-component bounded functions. We need two preliminary results.
\begin{lemma}\label{lem5.1}
Let $b \in H^\infty(\mathbb{D})$ with $||b||_\infty=1$. Let $\mu$ be the Clark measure of $b$ at the value $\alpha = 1$. 
% $$
% \frac{1+b(z)}{1-b(z)}=\int_{\partial \mathbb{D}} \frac{\xi+z}{\xi-z}d\mu(\xi)+2i\frac{\Im\left( b(0)\right)}{|1-b(0)|^2}\ .
% $$
Then $P(\mu)=\overline{\text{spec}(b)}$, where $P(\mu):=$spt$(\mu)\setminus a(\mu)$.
\end{lemma}
\begin{proof}
We need to verify the two inclusions. If $\xi \notin \overline{\text{spec(b)}}$, the function $b$ extends analytically across an arc containing $\xi$. Then either $(1+b(z))/(1-b(z))$ is analytic through an arc $ I(\xi)\subset \partial \mathbb{D}$ containing $\xi$ or $b(\xi)=1$. In the first case, $\mu_a(I(\xi))=0$ and consequently, since $\mu_s$ is carried on $\{ t: b(t)=1\}\subset \partial \mathbb{D}\setminus I(\xi)$, we deduce $\mu(I(\xi))=0$, which implies that $\xi \notin $spt$(\mu)$. If, on the other hand, $b(\xi)=1$, then $\xi \in a(\mu)$. Consequently, in both the cases, $\xi \notin P(\mu)$.

Let us now assume $\xi \notin P(\mu)$. If $\xi \in a(\mu)$, then $b(z)$ has an analytic extension at a neighbourhood of $\xi$ and $b(\xi)=1$, which implies that $\xi \notin \overline{\text{spec}(b)}$. If on the other hand, $\xi \notin $spt$(\mu)$, then $(1+b(z))/(1-b(z))$ extends analytically at a neighbourhood $I(\xi)$ of $\xi$ and its real part is equal to zero at $I(\xi)$. For this reason, $|b(\xi)|=1$ at $I(\xi)$, which means that $I(\xi)\subset \partial\mathbb{D}\setminus \overline{\text{spec}(b)}$, which implies that, in both the cases, $\xi \notin \overline{\text{spec}(b)}$. 
\end{proof}

\begin{lemma}\label{lem5.2nou}
Let $\mu$ be a positive, finite Borel measure on the unit circle satisfying the assumptions b i) and b ii) in Theorem \ref{theo1.6}. Then there exists a constant $C>0$ such that  

\begin{equation}\label{eq5.4}
\bigg| \int_{\partial \mathbb{D}} \frac{d\mu(t)}{1-\bar{t}r\xi}-\int_{\{|t-\xi|>1-r\}}\frac{d\mu(t)}{1-\bar{t}\xi}\bigg|\leq C,  \quad  0<r<1 ,\ \xi \in P(\mu) .    \end{equation}
\end{lemma}
\begin{proof}
In order to obtain (\ref{eq5.4}), it is sufficient to prove that that there exists a constant $C_1 >0$ such that 
\begin{equation}\label{eq5.5}
    \bigg| \int_{\{|t-\xi|>1-r \} } \left( \frac{1}{1-\bar{t}r\xi}-\frac{1}{1-\bar{t}\xi}\right)d\mu(t)\bigg|+ \bigg|\int_{\{|t-\xi|<1-r\}}\frac{d\mu(t)}{1-\bar{t}r\xi}\bigg|\leq C_1\ .  
\end{equation}
For $ n \geq 0$ let $I(n)$ be the arc centered at $\xi$ of measure $2^{n+1} (1-r)$. We note that there exists a constant $C_2 >0$ such that
\begin{align}
\label{eq5.6}    \int_{\{|t-\xi|>1-r\}} \frac{1-r}{|1-\bar{t}r\xi||1-\bar{t}\xi|}d \mu(t) \leq& \sum_{n=0}^\infty \int_{I(n+1) \setminus I(n)} \frac{1-r}{2^{2n}(1-r)^2}d\mu(t)\\
\nonumber =& \sum_{n=0}^{\infty} \frac{\mu (I(n+1) \setminus I(n)) }{2^{2n}(1-r)}\leq C_2\ ,
\end{align}
because by \emph{b.i)} and \emph{b.ii)}, $\mu(I)\lesssim |I|$ if $I$ contains two different points of spt$(\mu)$. Similarly there exists a constant $C_3 >0$ such that 
\begin{equation}\label{eq5.7}
    \int_{\{|t-\xi| < 1-r\}}\frac{d\mu(t)}{|1-\bar{t}r\xi|}\leq \frac{1}{1-r}\mu\left( \{t:\ |t-\xi|<1-r\}\right)\leq C_3\ .
\end{equation}
\noindent
Consequently, (\ref{eq5.6}) and (\ref{eq5.7}) imply (\ref{eq5.5}), from which estimate (\ref{eq5.4}) follows.
\end{proof}

\noindent We are now ready to prove Theorem \ref{theo1.6}. Even if some parts of this proof follow the ideas of the analogue version for inner functions in \cite{Be}, for the sake of completeness we have decided to present them here again.

\begin{proof}[Proof of Theorem \ref{theo1.6}]
We start by proving that \emph{a)} implies \emph{b)}. We assume that there exists a one-component bounded function $b$ such that
$$
\frac{1+b(z)}{1-b(z)}=\int_{\partial \mathbb{D}} \frac{\xi+z}{\xi-z}d\mu(\xi)+2i\frac{\Im\left( b(0)\right)}{|1-b(0)|^2}\ , \quad z \in \mathbb{D} . 
$$
Since spec$(b)$ is closed due to Lemma \ref{lem2.2}, Lemma \ref{lem5.1} gives that $P(\mu)=$spec$(b)$. First we note that spt$(\mu_a) \subset \text{spec} (b)$. Indeed, if $I$ is an arc contained in $\partial \mathbb{D}\setminus$ spec$(b)$, then $b$ extends analytically through $I$ and $|b|=1$ on $I$. Therefore the real part of $(1+b(r\xi))/(1-b(r\xi))$ tends to $0$ as $r$ tends to $1$, for any $\xi \in I$. Hence  $\mu_a(I)=0$.

Due to Lemma \ref{lem2.4}, there exists a constant $0<C<1$ such that 
$$
\text{spec}(b)=\left\lbrace \xi \in \partial \mathbb{D}: \ \limsup_{r \to 1} |b(r\xi)|\leq C\right\rbrace. 
$$
Therefore there exists a constant $C_1>0$, such that 
$$
1/C_1\leq \Re\left( \frac{1+b(r\xi)}{1-b(r\xi)}\right)=\int_{\partial \mathbb{D}}\frac{1-r^2}{|t-r\xi|}d\mu(t)\leq C_1\ ,
$$
when $1>r>r_0(\xi)$ and $\xi \in $ spec$(b)$. Hence
$$
1/C_1\leq g(\xi)\leq C_1 
$$
for $m$-almost every $\xi \in $ spec$(b)$. Consequently, $\mu$ satisfies \emph{b.i)}.

\vspace{11 pt}\noindent
Since $\limsup_{r \to 1}|b(r\xi)|\leq C<1 $ for any $\xi \in $ spec$(b) = P(\mu)$, we deduce
$$
P(\mu)\cap \left\lbrace \xi \in \partial \mathbb{D}:\ \limsup_{r \to 1} \Re b(r\xi)=1\right\rbrace=\emptyset. 
$$
Since the singular measure $\mu_s$ is carried by the set $\left\lbrace \xi \in \partial \mathbb{D}:\ \lim_{r \to 1} b(r\xi)=1\right\rbrace$, we deduce that $\mu_s(P(\mu))=0$, which implies that $\mu_s$ is purely atomic.

Let $I=(\alpha,\beta)$ be a connected component of $\partial \mathbb{D}\setminus P(\mu)$. Since spec$(b)=P(\mu)$, the function $b$ extends analytically through $I$. Hence, the set $a(\mu):=\left\lbrace \xi \in I:\ b(\xi)=1\right\rbrace$ is discrete. Corollary \ref{cor4.1} says that 
$$
|b'(\xi)|^{-1}\cong \text{dist}\left( \xi, \{ z \in \mathbb{D}:\ |b(z)|\leq C\}\right)
$$
when $\xi \in \partial \mathbb{D}\setminus $spec$(b)$. Hence 
$$
\lim_{I\ni \xi \to \alpha} |b'(\xi)|=\infty. 
$$
Since, due to Lemma \ref{lem2.6}, $\xi b'(\xi) b(\xi)^{-1}\geq 0$, when $\xi \in I$, we deduce that 
$$
\lim_{I\ni \xi \to \alpha} \frac{\xi b'(\xi)}{b(\xi)}=+\infty \ .
$$
Therefore, there exist infinitely many different points $\xi_k \in I$, such that $\xi_k \to \alpha$ and $b(\xi_k)=1$. A similar argument replacing $\alpha$ by $\beta$ shows that the set $a(\mu)\cap I$ is infinite and it accumulates to both $\alpha$ and $\beta$. We deduce that every $\xi \in a(\mu)$ has two neighbours. 

In order to prove that $C^{-1}|\xi-\xi_\pm|\leq \mu \{\xi \} \leq C|\xi-\xi_\pm|$ when $\xi \in a(\mu)$ and $\xi_\pm$ are its neighbours, we apply an argument due to A. Baranov and K. Dyakonov \cite{BD}. Fix $\xi \in a (\mu)$.  We pick $t \in (\xi,\xi_+)$ such that $b(t)=-1$. Then, using Theorem \ref{theo1.4}, we get
$$
|\xi_+ -\xi|>|t-\xi|\gtrsim|b'(\xi)|^{-1}|b(t)-b(\xi)| =  2|b'(\xi)|^{-1}\ .
$$
Hence $\mu \{\xi\} \lesssim |\xi - \xi_+|$. Let $N$ be a  positive integer to be fixed later. We split the arc $(\xi,\xi_+)$ into $N$ pairwise disjoints sub-arcs $J_k=(\alpha_k,\beta_k)$, with $k=1, \dots, N,$ such that
$$
\text{Arg}(b(\beta_k))-\text{Arg}(b(\alpha_k))=\frac{2\pi}{N}\ , \quad k=1, \ldots, N . 
$$
Let $s_k \in (\alpha_k,\beta_k)$, such that $|b'(s_k)|=\min\left\lbrace |b'(s)|:\ s \in [\alpha_k,\beta_k]\right\rbrace$. Then
\begin{equation}\label{eq5.1}
    \frac{2\pi}{N}=\int_{\alpha_k}^{\beta_k}\frac{sb'(s)}{b(s)}dm(s)\geq |b'(s_k)||\beta_k-\alpha_k|\ ,   \quad k=1, \ldots, N . 
\end{equation}
when $k=1,\dots, N$. Due to Corollary \ref{cor4.1},
\begin{equation}\label{eq5.2}
    |\beta_k-\alpha_k|\leq C_1\frac{2\pi}{N}\text{dist}\left( s_k, \{ z \in \mathbb{D}: |b(z)|\leq C\}\right)\  , \quad k=1, \ldots, N . 
\end{equation}
We now take $t \in (\alpha_k,\beta_k)$ and observe that $$
|b'(t)|=\int_{\overline{\mathbb{D}}}\frac{1}{|1-\bar{t}z|^2}d\sigma_b(z)\ .
$$
We note that 
$
|t-z|\geq |s_k-z|-|s_k-t|\geq |s_k-z|-|\beta_k-\alpha_k|\ .
$
Consequently, if we choose $N>4C_1\pi$, (\ref{eq5.2}) gives us that
$$
|t-z|\geq\frac{1}{2}|s_k-z|\, \quad \text{ if } |b(z)| \leq C .  
$$
Since spt$(\sigma_b)\subset \overline{\left\lbrace z \in \mathbb{D}:\ |b(z)|\leq C\right\rbrace}$, we deduce
$$
|b'(t)|\leq 4|b'(s_k)|
$$
and, by using (\ref{eq5.1}), we deduce that
$$
|\xi_+-\xi|\lesssim \sum_{k=1}^{N}|\beta_k-\alpha_k|\leq C(N) \frac{1}{|b'(\xi)|}\ .
$$
This finishes the proof of \emph{b.ii)}. 

\vspace{11 pt}\noindent 
Finally we prove  \emph{b.iii}). An easy calculation (see \cite{CMR}), shows that 
\begin{equation}\label{eq5.3}
\frac{1}{1-b(z)}=\int_{\partial \mathbb{D}}\frac{d\mu(t)}{1-\bar{t}z}-\frac{||\mu||-1}{2}-i\frac{\Im \left(b(0)\right)}{|1-b(0)|^2} , \quad z \in \D . 
\end{equation}
% We first verify that there exists a universal constant $C_2>0$ such that
% \begin{equation}\label{eq5.4}
% \bigg| \int_{\partial \mathbb{D}} % \frac{d\mu(t)}{1-\bar{t}r\xi}-\int_{\{|t-\xi|>1-r\}}\frac{d\mu(t)}{1-\bar{t}\xi}\bigg|\leq C_2 %    
% \end{equation}
% for every $\xi \in$ supt$(\mu_a)$. In fact, in order to obtain (\ref{eq5.4}), it is % sufficient to prove that
% \begin{equation}\label{eq5.5}
%     \bigg| \int_{\{|t-\xi|>1-r \} } \left( \frac{1}{1-\bar{t}r\xi}-\frac{1}{1-\bar{t}\xi}\right)d\mu(t)\bigg|+ \bigg|\int_{\{|t-\xi|<1-r\}}\frac{d\mu(t)}{1-\bar{t}r\xi}\bigg|\leq C_2\ .  
% \end{equation}
% For $ n \geq 0$ let $I(n)$ be the arc centered at $\xi$ of measure $2^{n+1} (1-r)$. We note % that 
% \begin{align}
% \label{eq5.6}    \int_{\{|t-\xi|>1-r\}} \frac{1-r}{|1-\bar{t}r\xi||1-\bar{t}\xi|}d \mu(t) % \leq& \sum_{n=0}^\infty \int_{I(n+1) \setminus I(n)} \frac{1-r}{2^{2n}(1-r)^2}d\mu(t)\\
% \nonumber =& \sum_{n=0}^{\infty} \frac{\mu (I(n+1) \setminus I(n)) }{2^{2n}(1-r)}\leq C_3\ ,
% \end{align}
% because by \emph{b.i)} and \emph{b.ii)}, $\mu(I)\lesssim |I|$ if $I$ contains two different % points of spt$(\mu)$. Similarly 
% \begin{equation}\label{eq5.7}
%     \int_{\{|t-\xi| < 1-r\}}\frac{d\mu(t)}{|1-\bar{t}r\xi|}\leq \frac{1}{1-r}\mu\left( \{t:\ % |t-\xi|<1-r\}\right)\leq C_4\ .
% \end{equation}
% \noindent
% Consequently, (\ref{eq5.6}) and (\ref{eq5.7}) imply (\ref{eq5.5}), from which estimate % (\ref{eq5.4}) follows.

Using (\ref{eq5.3}), (\ref{eq5.4}) and the fact that $\limsup_{r\to 1}|b(r\xi)|\leq C<1$ for $\xi \in $ spec$(b)$, we obtain that $H^*(\mu)(\xi)$ is bounded when $\xi \in P(\mu)$. 

On the other hand, if $\xi \in a(\mu)$, we have
\begin{align*}
    H^*(\mu)(\xi)=&\lim_{z \to \xi} \left( \frac{1}{1-b(z)}-\frac{1}{b'(\xi)(1-\bar{\xi}z)}\right)+C\\
    =&\lim_{z \to \xi}  \frac{1}{(1-\bar{\xi}z)}\left( \frac{1-\bar{\xi}z}{1-b(z)}-\frac{1}{b'(z)}\right)+C\\
    =& -\frac{b''(\xi)}{b'(\xi)^2}+C\ ,
\end{align*}
which is uniformly bounded due Corollary \ref{cor1.6}. This estimate finishes the proof of \emph{b.iii)} and, consequently, of \emph{b)}.

\vspace{22 pt}\noindent
Let us now show that \emph{b)} implies \emph{a)}. Given the measure $\mu$ satisfying the assumption in \emph{b)}, we need to show that the function $b$ defined as
\begin{equation}\label{eq5.7.1}
\frac{1+b(z)}{1-b(z)}=\int_{\partial \mathbb{D}}\frac{t+z}{t-z}d\mu(t)\, \quad     z \in \mathbb{D} 
\end{equation}
is one-component. The main ingredient will be Corollary \ref{cor1.6}. We start by showing the first condition in  Corollary \ref{cor1.6}. Let $\xi \in$ spec$(b)$. If $\xi$ lies in the spectrum of the inner factor of $b$ then $\xi \in \overline{ \{z \in \D : |b(z)| < \varepsilon \}}$, for any $\varepsilon >0$. If $\xi \in \text{spt} (\mu_a)$ we will show that there exists a constant $C<1$ such that
$$
 \limsup_{r \to 1}|b(r \xi)|\leq C.
$$
We note the (\ref{eq5.7.1}) gives
\begin{equation}\label{eq5.8}
    \frac{1-|b(r\xi)|^2}{|1-b(r\xi)|^2}=\int_{\partial \mathbb{D}} \frac{1-r^2}{|t-r\xi|^2}d\mu(t)\geq \frac{\mu\left(\{ t:\ |t-\xi|<1-r\}\right)}{4(1-r)}\geq C_1\ ,
\end{equation}
according to assumptions \emph{b.i)} and \emph{b.ii)}. On the other hand, \emph{b.iii)} and (\ref{eq5.4}) give that there exists a constant $C_2 >0$ such that
\begin{equation}\label{eq5.9}
    \limsup_{r\to 1}\bigg| \int_{\partial\mathbb{D}}\frac{d\mu(t)}{1-\bar{t}r\xi}\bigg|\leq C_2\ .
\end{equation}
Hence (\ref{eq5.3}), (\ref{eq5.8}) and (\ref{eq5.9}) provide a constant $0<C_3 < 1$ such that 
$$
\limsup_{r\to 1}|b(r\xi)|\leq C_3 \ .
$$
From this it follows that the function $b$ satisfies Condition $a)$ of Corollary \ref{cor1.6}. Next we will show that condition (b) in Corollary \ref{cor1.6} holds.  We need the following preliminary result. 
\begin{lemma}\label{lem5.2}
Let $\mu$ be a Borel measure which satisfies condition \emph{b)} of Theorem \ref{theo1.6} and let $\xi_0 \in a(\mu)$. Let $k>0$ be a constant sufficiently small compared with the constant $C=C(\mu)$ appearing in condition \emph{b)} of Theorem \ref{theo1.6}. We denote 
\begin{equation}\label{eq5.10}
    D_{\xi_0}(k):=\left\lbrace z \in \mathbb{C}:\ |z-\xi_0|<k\mu(\xi_0)\right\rbrace\ .
\end{equation}
Then, the function $b$ defined in (\ref{eq5.7.1}) has an analytic extension at $D_{\xi_0}(k)$ and 
\begin{equation}\label{eq5.11}
    \bigg| \frac{1-b(z)}{\xi_0-z}\bigg|\cong \frac{1}{\mu(\xi_0)}\ ,  \quad   z \in D_{\xi_0}(k) .
\end{equation}
\end{lemma}
\begin{proof}[Proof of Lemma \ref{lem5.2}]
We know that there exists a constant $C_1$ such that 
\begin{equation}\label{eq5.12}
    \frac{1}{1-b(z)}=\int_{\partial \mathbb{D}\setminus \{\xi_0\}} \frac{d\mu(t)}{1-\bar{t}z}+\frac{\mu(\xi_0)}{1-\bar{\xi_0}z}+C_1\ , \quad z \in \D . 
\end{equation}
We note that 
\begin{align*}
    \bigg| \int_{\partial \mathbb{D}\setminus \{\xi_0\}} \frac{d\mu(t)}{1-\bar{t}z}\bigg| &=\bigg| \int_{\partial \mathbb{D}\setminus \{\xi_0\}}\frac{1-\bar{t}\xi_0}{1-\bar{t}z}\frac{d\mu(t)}{1-\bar{t}\xi_0} \bigg| \\
    &\leq \bigg|\int_{\partial \mathbb{D}\setminus \{\xi_0\}} \frac{d\mu(t)}{1-\bar{t}\xi_0}\bigg|+\bigg| \int_{\partial \mathbb{D}\setminus \{\xi_0\}} \frac{\bar{t}(z-\xi_0)}{(1-\bar{t}z)(1-\bar{t}\xi_0)}d\mu(t)\bigg|\ .
\end{align*}
Due to \emph{b.iii)}, the first integral is uniformly bounded. The second integral is estimated as follows,   
$$
\bigg| \int_{\partial \mathbb{D}\setminus \{\xi_0\}} \frac{\bar{t}(z-\xi_0)}{(1-\bar{t}z)(1-\bar{t}\xi_0)}d\mu(t)\bigg|\leq |z-\xi_0|\int_{\partial\mathbb{D}\setminus \{\xi_0\}}\frac{d\mu(t)}{|1-\bar{t}z||1-\bar{t}\xi_0|} . 
$$
Since $|z-\xi_0|\leq k\mu(\xi_0)$ and $|t-\xi_0|\geq |\xi_{\pm}-\xi_0|\cong{\mu(\xi_0)}$, we deduce that $|t-z|\geq \frac{1}{2}|t-\xi_0|$ if $k$ is sufficiently small. Hence 
$$
|z-\xi_0|\int_{\partial\mathbb{D}\setminus \{\xi_0\}}\frac{d\mu(t)}{|1-\bar{t}z||1-\bar{t}\xi_0|}\leq 2 |z-\xi_0|\int_{\partial\mathbb{D}\setminus \{\xi_0\}}\frac{d\mu(t)}{|1-\bar{t}\xi_0|^2}
$$
Using the fact that $\mu(I)$ is comparable to $|I|$ when the arc $I$ contains two different points of spt$(\mu)$, we obtain that
$$
\int_{\partial \mathbb{D}\setminus \{\xi_0\}}\frac{d\mu(t)}{|t-\xi_0|^2}\cong \frac{1}{|\xi_+-\xi_0|}\ .
$$
Consequently there exists a constant $C_2 >0$ such that 
$$
\bigg| \int_{\partial \mathbb{D}\setminus \{\xi_0\}} \frac{\bar{t}(z-\xi_0)}{(1-\bar{t}z)(1-\bar{t}\xi_0)}d\mu(t)\bigg|\leq C_2 \ .
$$
At this point, identity (\ref{eq5.12}) implies that there exits a constant $C_3 >0$ such that 
$$
\bigg| \frac{\xi_0-z}{1-b(z)}-\mu(\xi_0)\bar{\xi_0}\bigg| \leq C_3 |\xi_0-z|\ .
$$
By taking $k$ sufficiently small so that $C_3|\xi_0-z|\leq C_3k\mu(\xi_0)\leq 1/2\  \mu(\xi_0)$, we deduce that
$$
\bigg| \frac{\xi_0-z}{1-b(z)}-\mu(\xi_0)\bar{\xi_0}\bigg| \leq \frac{1}{2}\mu(\xi_0)\ ,
$$
for every $z \in D_{\xi_0}(k)$, which proves the Lemma.
\end{proof}
\noindent
We can now check condition (b) of Corollary 1.6. The identity (\ref{eq5.7.1}) gives 
\begin{equation}\label{eq5.13}
\frac{b'(z)}{(1-b(z))^2}=\int_{\partial\mathbb{D}} \frac{ \bar{\xi}}{(1-\bar{\xi}z)^2}d\mu(\xi) , \quad z \in \D
\end{equation}
and
\begin{equation}\label{eq5.14}
\frac{b''(z)}{(1-b(z))^2}+\frac{{b'(z)}^2}{(1-b(z))^3}=2\int_{\partial\mathbb{D}} \frac{ \bar{\xi}^2}{(1-\bar{\xi}z)^3}d\mu(\xi)\  , \quad z \in \D.
\end{equation}
Let $\xi_0 \in a (\mu)$. By Lemma \ref{lem5.2}, 
$$
|1-b(z)|\gtrsim 1
$$
when $z \in \partial D_{\xi_0}(k)$. Hence, (\ref{eq5.13}) gives that
$$
|b'(z)|\lesssim \frac{\mu(\xi_0)}{|1-\bar{\xi_0}z|^2}+\int_{\partial\mathbb{D}\setminus\{\xi_0\}}\frac{d\mu(\xi)}{|1-\bar{\xi} z|^2}
$$
when $z \in \partial D_{\xi_0}(k)$. Since $\mu(I)\cong |I|$ if $I$ is an arc containg two points of spt$(\mu)$, we deduce that the last integral can be estimated by $|\xi_0-\xi_+|^{-1 } \cong 1/ \mu(\xi_0)$. Therefore
$$
|b'(z)|\lesssim \frac{1}{\mu(\xi_0)}
$$
when $z \in \partial D_{\xi_0}(k)$.
Now (\ref{eq5.14}) implies that 
$$
|b''(z)|\lesssim |b'(z)|^2+\frac{2\mu(\xi_0)}{|1-\bar{\xi_0}z|^3}+2\int_{\partial \mathbb{D}\setminus \{\xi_0\}}\frac{d\mu(\xi)}{|1-\bar{\xi}z|^3}\ , \quad z \in \partial D_{\xi_0}(k). 
$$
Acting as before, we obtain that 
$$
|b''(z)|\lesssim \frac{1}{\mu(\xi_0)^2}
$$
when $z \in \partial D_{\xi_0}(k)$, and hence, by Maximum Principle
$$
|b''(z)|\lesssim |b'(\xi_0)|^2
$$
for any $z \in D_{\xi_0}(k)$. Now identity (\ref{eq5.13}) gives 
$$
|b'(z)|=|1-b(z)|^2\bigg| \frac{\bar{\xi_0}}{(1-z\bar{\xi_0})^2}\mu(\xi_0)+\int_{\partial\mathbb{D}\setminus \{\xi_0\}} \frac{\bar{\xi}}{(1-\bar{\xi}z)^2}d\mu(\xi)\bigg|\ ,
$$
for any $z \in \partial \mathbb{D} \setminus \text{spec} (b)$. Applying Lemma \ref{lem5.2}, we deduce 
\begin{equation}\label{eq5.15}
|b'(z)| \gtrsim  \frac{\mu(\xi_0)}{|1-\bar{\xi_0}z|^2}|1-b(z)|^2\cong \frac{1}{\mu(\xi_0)}
\end{equation}
when $z \in D_{\xi_0}(k)$. Consequently
\begin{equation}\label{eq5.16}
|b''(z)|\leq |b'(z)|^2
\end{equation}
for any $z \in D_{\xi_0}(k)$. 

We need to obtain the same estimate also for the other points $z \in \partial\mathbb{D}\setminus\left( \text{spec}(b)  \cup_{a(\mu)}D_\xi(k)\right)$. We start by showing that
\begin{equation}\label{eq5.17}
|b(z)-1|>k
\end{equation}
for every $z \in \partial\mathbb{D}\setminus\left( \text{spec}(b)  \cup_{a(\mu)}D_\xi(k)\right)$. Indeed, let $z \in (\xi_-,\xi_+)$ where $\xi_\pm$ are neighbour atoms of $\mu$. We note that, due to Lemma \ref{lem5.2}, $|b(w)-1|\cong k$ if $w \in \partial D_{\xi_{\pm}}(k)\cap \partial\mathbb{D}$. Since Arg$(b)$ is increasing for $t \in (\xi_-,\xi_+)$, estimate (\ref{eq5.17}) follows. Therefore if $z \notin \cup_{a(\mu)}D_\xi(k) $, we have $|z-\xi_{\pm}|\geq k\mu(\xi_\pm)$. On the other hand, $|z-\xi_\pm|\leq |\xi_+-\xi_-|\cong\mu(\xi_{\pm})$. Therefore
\begin{equation}\label{eq5.18}
k\mu(\xi_\pm)\leq |z-\xi_\pm|\leq C\mu(\xi_\pm)\ .
\end{equation}
Applying the previous argument with $\xi_0$ replaced by $\xi_\pm$, we deduce that $|b''(z)|\lesssim |b'(z)|^2$. This finishes the proof of property \emph{b)} of Corollary \ref{cor1.6}. Finally we need to verify that $|b'(z)|\to \infty$ as $0<$dist$(z,\text{spec}(b))\to 0$. We note that (\ref{eq5.15}) and assumption \emph{b.ii)} gives it for $z \in \cup_{a(\mu)}D_{\xi}(k)$. For $z \notin \cup_{a(\mu)}D_{\xi}(k) $, let, as before, $z \in (\xi_-,\xi_+)$ where $\xi_\pm$ are neighbour atoms of $\mu$. Then  (\ref{eq5.17}), (\ref{eq5.18}) and (\ref{eq5.15}) give that $|b'(z)|\gtrsim |b'(\xi_{\pm})|$ which, since dist$(z,\text{spec}(b))\to 0$, finishes the proof.
\end{proof}

\begin{remark}
We note that if $b$ is an inner function, Theorem \ref{theo1.6} reduces to the result of \cite{Be} on Clark measures of one-component inner functions. Indeed, in this case, the absolute continuous part of the measure $\mu$ would be identically zero and the Hilbert transform $H^*(\mu)$ would be equal to
$$
H(\mu)(\xi):=\int_{\partial \mathbb{D}\setminus \{\xi\}} \frac{d\mu(t)}{1-\bar{t}\xi}\ ,
$$
for any $\xi \in a(\mu)$.
\end{remark}

%%%%%%%%%%%%%%%%%%%%%%%%%%%%%%%%%%%%%%%%%%%%%%%%%%%%%%%%%%%%%%%%%%%%%%%%%%%%%%
%%%%%%%%%%%%%%%%%%%%%%%%%%%%%%%%%%%%%%%%%%%%%%%%%%%%%%%%%%%%%%%%%%%%%%%%%%%%%%

\Addresses
\end{document}